\documentclass{article}

\usepackage{booktabs, multirow, bm}
\usepackage{amsmath, amssymb, amsthm, mathrsfs}
\usepackage{indentfirst}
\usepackage{graphicx, tikz, subfigure}
\usepackage{geometry}
\usepackage{longtable}
\usepackage{caption}
\usepackage{multirow}
\numberwithin{equation}{section}
\usepackage{xcolor}
\usepackage{ulem}
\usepackage[modulo]{lineno} 
\usepackage{hyperref}
 \usepackage{subcaption}
\allowdisplaybreaks[4]
\usepackage{cleveref}
\crefname{section}{Section}{Sections}
\crefname{figure}{Figure}{Figures}
\crefname{table}{Table}{Tables}
\crefname{equation}{}{}
\crefname{theorem}{Theorem}{Theorems}
\crefname{lemma}{Lemma}{Lemmas}
\crefname{remark}{Remark}{Remarks}
\crefname{problem}{Problem}{Subproblems}

\newtheorem{theorem}{Theorem}[section]

\newtheorem{remark}{Remark}[section]
\newtheorem{lemma}{Lemma}[section]
\newtheorem{proposition}{Proposition}[section]

\theoremstyle{definition}

\crefname{ip}{Co-inversion Problem}{ips}
\newtheoremstyle{MyThmStyle}
{}
{}
{}
{}
{\bfseries}
{}
{ }
{\thmname{#1\thmnumber{ #2\hspace{0.5em}}}\thmnote{(#3)}}
\theoremstyle{MyThmStyle}

\crefname{subisp}{inverse source problem}{ips}
\crefname{subiop}{inverse obstacle problem}{ips}
\graphicspath{{figurenew/}}
\definecolor{bananamania}{rgb}{0.98, 0.91, 0.71}

\begin{document}
	
	\title{Two operator splitting methods for three-dimensional stochastic Maxwell equations with multiplicative noise}

	\author{
		Liying Zhang\thanks{School of Mathematical Science, China University of Mining and Technology, Beijing 100083, China, lyzhang@lsec.cc.ac.cn }, Xinyue Kang\thanks{School of Mathematical Science, China University of Mining and Technology, Beijing 100083, China, SQT2300702045@student.cumtb.edu.cn },
	Lihai Ji\thanks{
		Institute of Applied Physics and Computational Mathematics, Beijing 100094, China, jilihai@lsec.cc.ac.cn} \thanks{Shanghai Zhangjiang Institute of Mathematics, Shanghai 201203, China, jilihai@lsec.cc.ac.cn}}
	
	\date{}
	\maketitle

\begin{abstract}
	In this paper, we develop two energy-preserving splitting methods for solving three-dimensional stochastic Maxwell equations driven by multiplicative noise. We use operator splitting methods to decouple stochastic Maxwell equations into simple one-dimensional subsystems and construct two stochastic splitting methods, Splitting Method I and Splitting Method II, through a combination of spatial compact difference methods and the midpoint rule in time discretization for the deterministic parts, and exact unitary analytical solutions for the stochastic parts. Theoretical proofs show that both methods strictly preserve the discrete energy conservation law. Finally, numerical experiments fully verify the energy conservation of the methods and demonstrate that the temporal convergence order of the two splitting methods is first-order.
\end{abstract}
	


\section{Introduction}
In this paper, we consider the following stochastic Maxwell equations
\begin{equation}\label{sto_maxwell_equations}
	\begin{split}
		&\varepsilon{\rm d}{\bf E}={\rm \bf curl}~{\bf H}{\rm d}t-\lambda{\bf H}\circ{\rm d}W,~(t,{\bf x})\in(0,~T]\times D,\\[1.5mm]
		&\mu{\rm d}{\bf H}=-{\rm \bf curl}~{\bf E}{\rm d}t+\lambda{\bf E}\circ{\rm d}W,~(t,{\bf x})\in(0,~T]\times D,
	\end{split}
\end{equation}
where $\circ$ denotes the Stratonovich integral, $D \subset \mathbb{R}^{3}$ is a bounded domain, $T \in (0, \infty)$, $\mathbf{E} = (E_1, E_2, E_3)^{\top}$ and $\mathbf{H} = (H_1, H_2, H_3)^{\top}$ are the electric and magnetic fields, and constants $\varepsilon$ and $\mu$ represent the electric permittivity and magnetic permeability, respectively. Here, $W(t)$ is a $Q$-Wiener process with respect to a filtered probability space $(\Omega, \mathcal{F}, \{\mathcal{F}_{t}\}_{0 \leq t \leq T}, \mathbb{P})$, where $Q$ is a symmetric and positive definite operator with finite trace on $U = L^{2}(D)$.  
Let $\{e_i\}_{i \in \mathbb{N}}$ denote an orthonormal basis of $U$. Then $W(t)$ can be represented as  
\begin{equation}\label{W}  
W(t) = \sum_{i=1}^{\infty} Q^{\frac{1}{2}} e_i \beta_i(t), \quad t \in [0, T],  
\end{equation}  
where $\{\beta_i(t)\}_{i \in \mathbb{N}}$ are independent real-valued Brownian motions.

Three-dimensional stochastic Maxwell equations with multiplicative noise play a critical role in many scientific fields \cite{CC1,LL1,MM1,SS1}. To date, scholars have conducted studies on the existence and uniqueness of solutions to the three-dimensional stochastic Maxwell equations, as well as certain properties of these solutions in the literature. Due to the equivalence between the Stratonovich and It\^o integrals and by \cite[Theorem 2.1]{CC2}, the stochastic Maxwell Eq. \eqref{sto_maxwell_equations} admits a unique mild solution, as also discussed in, for example, \cite{CC3,ZZ1}. Furthermore, \cite{JJ1} proved that the system energy of three-dimensional stochastic Maxwell equations with multiplicative noise is almost surely conserved.

However, since obtaining exact solutions for three-dimensional stochastic Maxwell equations is quite difficult, studying their numerical methods plays a crucial role in practical applications. \cite{CC2} constructed a semi-implicit Euler method for stochastic Maxwell equations with It\^o multiplicative noise, which exhibits a mean-square convergence order of $\frac{1}{2}$. Using regularity estimates of mild solutions, \cite{DD1} proved that the strong convergence order remains $\frac{1}{2}$ for general multiplicative noise. \cite{CC3} introduced a class of generalized stochastic Runge-Kutta methods for stochastic Maxwell equations with additive noise. Nevertheless, preserving the energy conservation property in numerical methods remains particularly essential. Currently, numerous studies have developed energy-conserving numerical methods for deterministic Maxwell equations, with various approaches proposed \cite{JJ2,WW1,WW2}. However, extending these deterministic methods to stochastic scenarios faces significant challenges: stochastic oscillations amplify numerical dissipation, destroying the system's energy conservation. To address this limitation, \cite{JJ1} proposed a stochastic energy-conserving method for additive noise cases, combining spatial wavelet collocation with temporal stochastic symplectic methods.

However, it has not yet been proven that stochastic splitting methods for Eq. \eqref{sto_maxwell_equations} preserve the energy conservation property. In this paper, we use operator splitting techniques to decouple Eq. \eqref{sto_maxwell_equations} into simple one-dimensional subsystems and construct two stochastic splitting methods, Splitting I and Splitting II, through a combination of spatial compact difference methods and the midpoint rule in time discretization for the deterministic parts, and exact unitary analytical solutions for the stochastic parts. Throughout the process, we strictly ensure that each substep maintains discrete energy invariance, ultimately guaranteeing unconditional strict conservation of electromagnetic energy by the entire algorithm.

The structure of this paper is as follows: Section \ref{preliminaries} presents theoretical results, including the well-posedness and energy conservation laws for Eq. \eqref{sto_maxwell_equations}. Section \ref{sect_3} details two operator splitting methods: Splitting method I, based on positive/negative decomposition of the curl operator, and Splitting method II, based on coordinate-direction splitting. This section decomposes the deterministic and stochastic subsystems and demonstrates the unitary group properties of the sub-operators. Section \ref{Energy conservation splitting  methods} constructs two splitting methods and proves their discrete energy conservation: Subsection \ref{Setting} establishes the notation used throughout this paper, and Subsection \ref{Splitting methods} proposes the corresponding numerical methods. These methods employ spatial compact difference discretizations combined with the temporal midpoint rule for deterministic subsystems and exact analytical solutions for stochastic subsystems. Section \ref{Conservation properties} rigorously proves the discrete energy conservation for both methods. Finally, Section \ref{Numerical experiments2} conducts numerical experiments to verify the theoretical results.

\section{Mathematical preliminaries}\label{preliminaries}
In this section, we lay out the mathematical framework for the article and summarize key results for stochastic Maxwell equations from the literature, including the existence and uniqueness of the mild solution and the energy conservation law.

The basic Hilbert space we work with is ${\mathbb H} = L^2(D)^3 \times L^2(D)^3$ with inner product defined by
\[
\left\langle \begin{pmatrix}
	{\bf E}_1\\{\bf H}_1
\end{pmatrix},~ \begin{pmatrix}
	{\bf E}_2\\{\bf H}_2
\end{pmatrix}\right\rangle_{\mathbb H}:=\int_{D}(\varepsilon {\bf E}_1\cdot {\bf E}_2
+\mu{\bf H}_1\cdot{\bf H}_2){\rm d}{\bf x}
\]
and the norm
\[
\left\|\begin{pmatrix}
	{\bf E}\\{\bf H}
\end{pmatrix}\right\|_{\mathbb H}=\left(\int_{D}(\varepsilon|{\bf E}|^2+\mu|{\bf H}|^2){\rm d}{\bf x}\right)^{\frac{1}{2}},\quad \forall~{\bf E},{\bf H}\in L^2(D)^3.
\]
We define the Maxwell operator by
\begin{equation}\label{M_operator}
	M=\begin{pmatrix}
		0& {\rm \bf curl} \\
		-{\rm \bf curl} &0 \\
	\end{pmatrix}
\end{equation}
with domain
\begin{equation*}
	\begin{split}
		{\mathcal D}(M)&=\left\{\begin{pmatrix}
			{\bf E} \\
			{\bf H}
		\end{pmatrix}\in {\mathbb H}:~M\begin{pmatrix}
			{\bf E} \\
			{\bf H}
		\end{pmatrix}=\begin{pmatrix}
			{\rm\bf curl}~{\bf H}\\
			-{\rm\bf curl}~{\bf E}
		\end{pmatrix}\in{\mathbb H},~ {\bf n}\times{\bf E}\Big|_{\partial D}=0 \right\}\\[1.5mm]
		&=H_0({\rm\bf curl},D)\times H({\rm\bf curl},D).
	\end{split}
\end{equation*}
Let $B:[0,T]\times \mathbb{H}\to HS(U_0,{\mathbb H})$ be a Nemytskij operator defined by
\begin{equation}\label{B}
	(B(t,u)v)({\bf x})=\begin{pmatrix}
		-\lambda {\bf H}(t,{\bf x}))v({\bf x})\\
		\lambda{\bf E}(t,{\bf x}))v({\bf x})
	\end{pmatrix},\quad {\bf x}\in D,
\end{equation}
where $u=({\bf E}^T,{\bf H}^T)^T\in{\mathbb H}$, and $v\in U_0:=Q^{\frac12}U$. To ensure the existence of a unique mild solution for
Eq. \eqref{sto_maxwell_equations} over the interval $[0,T]$, it suffices to suppose that the coefficients $\varepsilon$ and $\mu$ satisfy
\begin{equation}\label{let_coeff}
	\varepsilon,~\mu\in L^{\infty}(D),~~\varepsilon,~\mu\geq\delta>0
\end{equation}
for a positive constant $\delta>0$.

Using these notations, we obtain the abstract form of stochastic Maxwell equations
in infinite-dimensional space $\mathbb{H}$:
\begin{equation}\label{sto_maxwell_equations_abstract}
	\begin{split}
		A{\rm d}Z(t)&=MZ(t){\rm d}t+B(t,Z(t))\circ{\rm d}W(t),\quad t\in(0,~T],\\
		Z(0)&=Z_0,
	\end{split}
\end{equation}
where $Z=({\bf E}^{\top},{\bf H}^{\top})^{\top}$, $Z_0=({\bf E}_0^{\top},{\bf H}_0^{\top})^{\top}$ and $A=\left(
\begin{array}{cc}
	\varepsilon I & 0 \\
	0 & \mu I \\
\end{array}
\right)$.

Due to the equivalence between Stratonovich integral and It\^o integral and follows from \cite[Theorem 2.1]{CC2}, stochastic Maxwell Eq. \eqref{sto_maxwell_equations_abstract} have a unique mild solution, which has been also discussed, for example, in \cite{CC3,ZZ1}.

We begin with the following invariant of the electromagnetic energy functional
\begin{equation}\label{energy_con}
	\mathcal{H}(Z(t)):=\int_{D}\left(\varepsilon|{\bf E}(t,{\bf x})|^2+\mu|{\bf H}(t,{\bf x})|^{2}\right){\rm d}{\bf x}.
\end{equation}
\begin{proposition}
	Under the perfectly electric conducting boundary condition ${\bf n}\times{\bf E}=0$ on $\partial D$, we have
	\begin{equation}
		\mathcal{H}(Z(t))=\mathcal{H}(Z_0),\quad{\mathbb P}\text{-}a.s.,
	\end{equation}
	for any $t\in[0, T]$.
\end{proposition}

We refer the reader to Theorem 3.3 in \cite{ZZ1} for a similar proof of the proposition.
\section{Two operator splitting methods}\label{sect_3}
	To address the memory difficulties associated with conventional numerical methods for the three-dimensional stochastic Maxwell equations, we apply a local one-dimensional idea to propose two operator splitting methods for the stochastic Maxwell equations \eqref{sto_maxwell_equations}. In this section, we also present several properties of the proposed methods.
			
			We note that the ${\rm\bf curl}$ operator can be split into 
			\begin{equation}
				{\rm\bf curl}={\rm\bf curl_+}+{\rm\bf curl_{-}},
			\end{equation}
			with 
			\begin{equation*}
				{\rm\bf curl_+}=\left[
				\begin{array}{ccc}
					0 & 0 & \partial_y \\
					\partial_z & 0 & 0 \\
					0 & \partial_x & 0 \\
				\end{array}
				\right],\quad {\rm\bf curl_{-}}=\left[
				\begin{array}{ccc}
					0 & -\partial_z & 0 \\
					0 & 0 & -\partial_x \\
					-\partial_y & 0 & 0 \\
				\end{array}
				\right].
			\end{equation*}
			and 
			\begin{equation}
				{\rm\bf curl}={\rm\bf curl_x}+{\rm\bf curl_{y}}+{\rm\bf curl_{z}},
			\end{equation}
			with 
			\begin{equation*}
				{\rm\bf curl_x}=\left[
				\begin{array}{ccc}
					0 & 0 & 0 \\
					0 & 0 & -\partial_x \\
					0 & \partial_x & 0 \\
				\end{array}
				\right], {\rm\bf curl_{y}}=\left[
				\begin{array}{ccc}
					0 & 0 & \partial_y \\
					0 & 0 & 0 \\
					-\partial_y & 0 & 0 \\
				\end{array}
				\right],{\rm\bf curl_{z}}=\left[
				\begin{array}{ccc}
					0 & -\partial_z & 0 \\
					\partial_z & 0 & 0 \\
					0 & 0 & 0 \\
				\end{array}
				\right].
			\end{equation*}
			respectively.
			
			Define operators
			\begin{equation}\label{M_1}
				M_{+}=\left[
				\begin{array}{cc}
					{\bf 0} & {\rm\bf curl_{+}} \\
					-{\rm\bf curl_{+}} & {\bf 0} \\
				\end{array}
				\right],\quad M_{-}=\left[
				\begin{array}{cc}
					{\bf 0} & {\rm\bf curl_{-}} \\
					-{\rm\bf curl_{-}} & {\bf 0} \\
				\end{array}
				\right],
			\end{equation}
			in $\mathbb{H}$ and endowed with domains
			\begin{equation*}
				\begin{split}
					{\mathcal D}(M_+)=\left\{\begin{pmatrix}
						{\bf E} \\
						{\bf H}
					\end{pmatrix}\in {\mathbb H}:~M_+\begin{pmatrix}
						{\bf E} \\
						{\bf H}
					\end{pmatrix}\in{\mathbb H},~ {\bf n}\times{\bf E}=0~{\texttt on}~\partial D \right\},\\
					{\mathcal D}(M_-)=\left\{\begin{pmatrix}
						{\bf E} \\
						{\bf H}
					\end{pmatrix}\in {\mathbb H}:~M_-\begin{pmatrix}
						{\bf E} \\
						{\bf H}
					\end{pmatrix}\in{\mathbb H},~ {\bf n}\times{\bf E}=0~{\texttt on}~\partial D \right\}.
				\end{split}
			\end{equation*}
			It can be verified that ${\mathcal D}(M_+)\cap{\mathcal D}(M_-)\subset{\mathcal D}(M)$ and $M_+u+M_-u=Mu$ for $u\in{\mathcal D}(M_+)\cap{\mathcal D}(M_-)$. In a similar manner, we define operators
			\begin{equation}\label{M_2}
				M_{\alpha}=\left[
				\begin{array}{cc}
					{\bf 0} & {\rm\bf curl_{\alpha}} \\
					-{\rm\bf curl_{\alpha}} & {\bf 0} \\
				\end{array}
				\right],\quad \alpha=x,y,z,
			\end{equation}
			in $\mathbb{H}$ and endowed with domains
			\begin{equation*}
				\begin{split}
					{\mathcal D}(M_x)=\left\{\begin{pmatrix}
						{\bf E} \\
						{\bf H}
					\end{pmatrix}\in {\mathbb H}:~M_x\begin{pmatrix}
						{\bf E} \\
						{\bf H}
					\end{pmatrix}\in{\mathbb H},~E_2=E_3=0~{\texttt on}~\Gamma_1^{\pm} \right\},\\
					{\mathcal D}(M_y)=\left\{\begin{pmatrix}
						{\bf E} \\
						{\bf H}
					\end{pmatrix}\in {\mathbb H}:~M_y\begin{pmatrix}
						{\bf E} \\
						{\bf H}
					\end{pmatrix}\in{\mathbb H},~E_1=E_3=0~{\texttt on}~\Gamma_2^{\pm}  \right\},\\
					{\mathcal D}(M_z)=\left\{\begin{pmatrix}
						{\bf E} \\
						{\bf H}
					\end{pmatrix}\in {\mathbb H}:~M_z\begin{pmatrix}
						{\bf E} \\
						{\bf H}
					\end{pmatrix}\in{\mathbb H},~E_1=E_2=0~{\texttt on}~\Gamma_3^{\pm}  \right\},
				\end{split}
			\end{equation*}
			where $\Gamma_1^{\pm}\cup\Gamma_2^{\pm}\cup\Gamma_3^{\pm}=\partial D$. We also have ${\mathcal D}(M_x)\cap{\mathcal D}(M_y)\cap{\mathcal D}(M_z)\subset{\mathcal D}(M)$ and $M_xu+M_yu+M_zu=Mu$ for $u\in{\mathcal D}(M_x)\cap{\mathcal D}(M_y)\cap{\mathcal D}(M_z)$.
			\begin{lemma}\label{operator_lemma}
				The operators $M_+,M_-$ defined in \eqref{M_1} and $M_\alpha$ $(\alpha=x,y,z)$ defined in \eqref{M_2} are skew-adjoint on $\mathbb{H}$.
			\end{lemma}
			\begin{proof}
				We only present the proof for $M_\alpha$ $(\alpha=x,y,z)$ here, since the proof for $M_+,M_-$ is similar.

                To prove $M_\alpha^{\ast} = -M_\alpha$, it suffices to show that $M_\alpha$ is skew-symmetric and that $\mathrm{Id}\pm M_\alpha$ has dense range. To establish the skew-symmetry of $M_x$, take $\psi = (u, v)$ and $\tilde{\psi} = (\tilde{u}, \tilde{v})$ in $\mathcal{D}(M_x)$. The integration by parts formula then yields:
				\begin{align*}
					\left\langle M_x\psi, \tilde{\psi}\right\rangle_\mathbb{H}
					&=\int_{D} \left(-\tilde{u}_2\partial_x v_3+\tilde{u}_3\partial_x v_2 +\tilde{v}_2\partial_x u_3-\tilde{v}_3\partial_x u_2 \right){\rm d}{\bf x}\\[1.5mm]
					&=\int_{D} \left(v_3 \partial_x\tilde{u}_2-v_2\partial_x\tilde{u}_3-u_3\partial_x\tilde{v}_2+u_2 \partial_x\tilde{v}_3\right){\rm d}{\bf x}\\[1.5mm]
					&=-\left\langle \psi,  M_x\tilde{\psi}\right\rangle_\mathbb{H},
				\end{align*}
				due to the zero boundary conditions by the definition of ${\mathcal D}(M_x)$. Thus $M_x$ is skew-symmetric and analogously for $M_y,M_z$.
				
			To establish the density of $\mathrm{Id}\pm M_x$, it suffices to demonstrate that
				\begin{equation}\label{1.42}
					\overline{{\rm ran}(Id\pm M_x)}=\mathbb{H}.
				\end{equation}
				Because $C^{\infty}(D)$ is dense in $\mathbb{H}$, we infer that \eqref{1.42} is equivalent to showing that for every $f = (\varphi^\top, \psi^\top)^\top \in C^{\infty}(D)$, there exists $g = (\mathbf{E}^\top, \mathbf{H}^\top)^\top \in \mathcal{D}(M_x)$ such that
				\begin{equation}\label{eq 4.1}
					(Id\pm M_x)g=f,
				\end{equation}
				or equivalently,
				\begin{equation*}
					\begin{split}
						&E_1=\varphi_1,\quad E_2 \mp \partial_x H_3 =\varphi_2,\quad E_3\pm \partial_x H_2=\varphi_3,\\[1.5mm]
						&H_1=\psi_1,\quad H_2\mp \partial_x E_3=\psi_2,\quad H_3\pm \partial_x E_2=\psi_3.
					\end{split}
				\end{equation*}
				It yields
				\begin{align*}
					E_2-\partial_{xx} E_2=\varphi_2\pm \partial_x \psi_3=:\hat{f}_2\in L^2(D),\\[1.5mm]
					E_3 -\partial_{xx} E_3=\varphi_3\mp \partial_x\psi_2=:\hat{f}_3\in L^2(D).
				\end{align*}
				In order to solve these equations, we introduce the operator $D_x=\partial_{xx}$ with domain
				\[
				{\mathcal D}(D_x)=\{u\in L^2(D):~~\partial_x u,~D_x u\in L^2(D),~~u=0 \mbox{ on } \Gamma_1^{\pm} \}.
				\]
				The Lax-Milgram lemma thus yields the existence of $E_2$, $E_3\in {\mathcal D}(D_x)$ such that $E_j-D_x E_j=\hat{f}_j, j=2,3$ holds. If we now define
				\[
				H_2=\psi_2\pm\partial_x E_3,\qquad H_3=\psi_3\mp\partial_x E_2,
				\]
				we obtain a solution $g=({\bf E}^T,{\bf H}^T)^T\in {\mathcal D}(M_x)$ of \eqref{eq 4.1}, as asserted. Similarly, we can get the results for $M_y$ and $M_z$.
			\end{proof}
			
			Now we are in the position to present the operator splitting of stochastic Maxwell equations. We decompose \eqref{sto_maxwell_equations_abstract} into two deterministic subsystems
			\begin{equation}\label{17}
				\begin{split}
					A{\rm d}Z(t)=M_+Z(t){\rm d}t,\quad A{\rm d}Z(t)=M_-Z(t){\rm d}t,
				\end{split}
			\end{equation}
			and a stochastic system
			\begin{align}\label{sto_system}
				A{\rm d}Z(t)=B(t,Z(t))\circ{\rm d}W(t).
			\end{align}
			Similarly, we can decompose \eqref{sto_maxwell_equations_abstract} into the following three deterministic subsystems
			\begin{equation}\label{19}
				\begin{split}
					A{\rm d}Z(t)=M_xZ(t){\rm d}t,\quad A{\rm d}Z(t)=M_yZ(t){\rm d}t,\quad A{\rm d}Z(t)=M_zZ(t){\rm d}t
				\end{split}
			\end{equation}
			and a stochastic system \eqref{sto_system}.
			\begin{remark}
				Under condition \eqref{let_coeff}, and based on Lemma \ref{operator_lemma} and by applying Stone's theorem (see, for instance, \cite[Theorem II.3.24]{EE1}), the operators $M_+$, $M_-$, $M_x$, $M_y$ and $M_z$ generate unitary $C_0$-groups $S_+(t)=e^{tM_+}$, $S_-(t)=e^{tM_-}$, $S_x(t)=e^{tM_x}$, $S_y(t)=e^{tM_y}$ and $S_z(t)=e^{tM_z}$ on $\mathbb{H}$ for $t\in\mathbb{R}$, respectively. Thus, the above two subsystems are both well-posedness.
			\end{remark}
			
			To show more clearly, we give the corresponding explicit componentwise forms of the above two splitting methods.
			\begin{itemize}
				\item[(1)] {\bf Splitting I.}
				\begin{itemize}
					\item Subsystem 1: 
					\begin{equation*}\label{20}
						\left\{\begin{array}{c}
							\varepsilon{\rm d}E_3=\partial_xH_2{\rm d}t\\
							\mu{\rm d}H_2=\partial_xE_3{\rm d}t
						\end{array}\right.,~\left\{\begin{array}{c}
							\varepsilon{\rm d}E_1=\partial_yH_3{\rm d}t\\
							\mu{\rm d}H_3=\partial_yE_1{\rm d}t
						\end{array}\right.,~{\rm and}~ \left\{\begin{array}{c}
							\varepsilon{\rm d}E_2=\partial_zH_1{\rm d}t\\
							\mu{\rm d}H_1=\partial_zE_2{\rm d}t
						\end{array}\right..
					\end{equation*}
					\item Subsystem 2:
					\begin{equation*}\label{21}
						\left\{\begin{array}{c}
							\varepsilon{\rm d}E_2=-\partial_xH_3{\rm d}t\\
							\mu{\rm d}H_3=-\partial_xE_2{\rm d}t
						\end{array}\right.,~\left\{\begin{array}{c}
							\varepsilon{\rm d}E_3=-\partial_yH_1{\rm d}t\\
							\mu{\rm d}H_1=-\partial_yE_3{\rm d}t
						\end{array}\right.,~{\rm and}~ \left\{\begin{array}{c}
							\varepsilon{\rm d}E_1=-\partial_zH_2{\rm d}t\\
							\mu{\rm d}H_2=-\partial_zE_1{\rm d}t
						\end{array}\right..
					\end{equation*}
					\item Subsystem 3:
					\begin{equation*}
						\varepsilon{\rm d}{\bf E}=-\lambda{\bf H}\circ{\rm d}W,\quad\mu{\rm d}{\bf H}=\lambda{\bf E}\circ{\rm d}W.
					\end{equation*}
				\end{itemize}     
			\end{itemize}
			\begin{itemize}
				\item[(2)] {\bf Splitting II.}
				\begin{itemize}
					\item Subsystem 1: 
					\begin{equation*}\label{3}
						\left\{\begin{array}{c}
							{\rm d}E_x=0\\
							{\rm d}H_x=0
						\end{array}\right.,~\left\{\begin{array}{c}
							\varepsilon{\rm d}E_y=-\partial_xH_z{\rm d}t\\
							\mu{\rm d}H_z=-\partial_xE_y{\rm d}t
						\end{array}\right., ~{\rm and}~ \left\{\begin{array}{c}
							\varepsilon{\rm d}E_z=\partial_xH_y{\rm d}t\\
							\mu{\rm d}H_y=\partial_xE_z{\rm d}t
						\end{array}\right..
					\end{equation*}
					\item Subsystem 2:
					\begin{equation*}\label{4}
						\left\{\begin{array}{c}
							{\rm d}E_y=0\\
							{\rm d}H_y=0
						\end{array}\right.,~\left\{\begin{array}{c}
							\varepsilon{\rm d}E_z=-\partial_yH_x{\rm d}t\\
							\mu{\rm d}H_x=-\partial_yE_z{\rm d}t
						\end{array}\right., ~{\rm and}~ \left\{\begin{array}{c}
							\varepsilon{\rm d}E_x=\partial_yH_z{\rm d}t\\
							\mu{\rm d}H_z=\partial_yE_x{\rm d}t
						\end{array}\right..
					\end{equation*}
					\item Subsystem 3: 
					\begin{equation*}\label{5}
						\left\{\begin{array}{c}
							{\rm d}E_z=0\\
							{\rm d}H_z=0
						\end{array}\right.,~\left\{\begin{array}{c}
							\varepsilon{\rm d}E_x=-\partial_zH_y{\rm d}t\\
							\mu{\rm d}H_y=-\partial_zE_x{\rm d}t
						\end{array}\right., ~{\rm and}~ \left\{\begin{array}{c}
							\varepsilon{\rm d}E_y=\partial_zH_x{\rm d}t\\
							\mu{\rm d}H_x=\partial_zE_y{\rm d}t
						\end{array}\right..
					\end{equation*}
					\item Subsystem 4: 
					\begin{equation*}
						\varepsilon{\rm d}{\bf E}=-\lambda{\bf H}\circ{\rm d}W,\quad\mu{\rm d}{\bf H}=\lambda{\bf E}\circ{\rm d}W.
					\end{equation*}
				\end{itemize}
			\end{itemize}
			
			\begin{remark}
				With the perfectly electric conducting boundary conditions imposed and by Lemma \ref{operator_lemma}, the energy conservation law still hold for {\bf Splitting I} (resp. {\bf Splitting II}).
			\end{remark}
			
			\section{Energy conservation splitting  methods}\label{Energy conservation splitting  methods}
            Drawing on the operator splitting techniques developed in Subsection \ref{sect_3}, we propose two distinct energy-conserving splitting methods for three-dimensional stochastic Maxwell equations with multiplicative noise. These methods combine a spatial compact difference method and temporal midpoint rule for \textbf{Splitting I} (resp. \textbf{Splitting II}) with exact analytical solutions of \eqref{sto_system}.
			\subsection{Setting}\label{Setting}
			In what follows, we consider a cubic domain $D=[x_L,x_R]\times[y_L,y_R]\times [z_L,z_R]$ in this paper. Let $h_x$, $h_y$ and $h_z$ be the mesh sizes along $x$, $y$ and $z$ directions, respectively, and $\tau$ is the time step size. The spatial-temporal domain $D_T$ is partitioned by $x_i=x_L+ih_x,y_j=y_L+jh_y,z_k=z_L+kh_z,t^n=n\tau$, for $i=0,1,\cdots,I;~j=0,1,\cdots,J;~k=0,1,\cdots,K$ and $n=0,1,\cdots,N$. For a function $u(x,y,z,t)$, let us define the following operator
			\begin{equation*}
				\delta_tu_{i,j,k}^{n}=\frac{u_{i,j,k}^{n+1}-u_{i,j,k}^{n}}{\tau}
			\end{equation*}
			and introduce four matrices
			\begin{equation*}\label{AB}
				\begin{split}
					&{A}=\frac{1}{2}\left(\begin{array}{cccccccc}
						2&1&&&1\\
						1&2&1&&\\
						&\ddots&\ddots&\ddots&\\
						&&1&2&1\\
						1&&&1&2
					\end{array}\right),\quad{B}=\left(\begin{array}{cccccccc}
						0&1&&&-1\\
						-1&0&1&&\\
						&\ddots&\ddots&\ddots&\\
						&&-1&0&1\\
						1&&&-1&0
					\end{array}\right)
				\end{split}
			\end{equation*}
			and 
			\begin{equation*}
				\begin{split}
					&\mathcal{A}=\left(\begin{array}{cc}
						{A}&0\\
						0&{A}
					\end{array}\right),\quad\mathcal{B}=\left(\begin{array}{cccccccc}
						0&B\\
						B&0
					\end{array}\right).
				\end{split}
			\end{equation*}
			
			The following lemma is very helpful to inherit the intrinsic properties of stochastic Maxwell equations.
			\begin{lemma}\label{lemma_1}
				Matrix $\mathcal{A}$ is symmetric and positive definite, while, $\mathcal{B}$ is skew-symmetric. Moreover, $\mathcal{A}^{-1}\mathcal{B}$ is skew-symmetric.
			\end{lemma}
			
			For sake of simplicity, we will use the following notation.
			\begin{itemize}
				\item[1.] Denote $\|\cdot\|_{l^2}$ the discrete norm of a scalar function, given by
				$$
				\|U^n\|_{l^2}:=\left(h_xh_yh_z\sum_{i=0}^{I-1}\sum_{j=0}^{J-1}\sum_{k=0}^{K-1}|U_{i,j,k}^n|^2\right)^{\frac{1}{2}}.
				$$
				\item[2.] Denote $S^n(t)$ a strongly continuous group generated by operator  
				\begin{equation}\label{dis_oper}
					A^n=\left[\begin{array}{cc}
						0&-\lambda\varepsilon^{-1}I_{3\times3}\\
						\lambda\mu^{-1}I_{3\times3}&0
					\end{array}\right]
				\end{equation}
				and 
				\begin{equation*}
					S^n(t)=\exp\left(tA^n\right)
					=\left[\begin{array}{cc}
						C^n(t)& -\varepsilon^{-\frac{1}{2}}\mu^{\frac{1}{2}}B^n(t) \\
						\varepsilon^{\frac{1}{2}}\mu^{-\frac{1}{2}}B^n(t) & C^n(t)
					\end{array}\right],
				\end{equation*}
				where $C^n(t)=\cos(\lambda t(\varepsilon\mu I)^{-\frac{1}{2}})$ and $B^n(t)=\sin(\lambda t(\varepsilon\mu I)^{-\frac{1}{2}})$ are the discrete cosine and sine operators, respectively. 
				\item[3.] The operator $A^n$ defined in \eqref{dis_oper} is skew-adjoint on $\mathbb{H}$ and thus generates a unitary $C_0$-group $S^n(t)=e^{tA^n}$ on $\mathbb{H}$ in view of Stone's theorem.
				
				\item[4.] Denote
				\begin{equation*}
					\begin{split}
						&u_{\cdot,j,k}=\left(u_{0,j,k},u_{1,j,k},\cdots,u_{I-1,j,k}\right)^{T},u^{[1]+n}=\frac{1}{2}\left(u^{[1]}+u^{n}\right),\\
						&u_{i,\cdot,k}=\left(u_{i,0,k},u_{0,1,k},\cdots,u_{i,J-1,k}\right)^{T},u^{[2]+[1]}=\frac{1}{2}\left(u^{[2]}+u^{[1]}\right),\\
						&u_{i,j,\cdot}=\left(u_{i,j,0},u_{i,j,1},\cdots,u_{i,j,K-1}\right)^{T},u^{[3]+[2]}=\frac{1}{2}\left(u^{[3]}+u^{[2]}\right).
					\end{split}
				\end{equation*}
				
			\end{itemize}
			
			\subsection{Splitting methods}\label{Splitting methods}
	Using the defined difference operators, matrices, and notation, we formulate two splitting methods for the stochastic Maxwell equations \eqref{sto_maxwell_equations} as follows.
    
			{\bf (1) Splitting method I}
			
			{\bf Stage 1.}  Compute intermediate variables ${\bf E}^{[1]},{\bf H}^{[1]}$ from ${\bf E}^{n}$ and ${\bf H}^{n}$:
			\begin{equation}\label{100}
				\begin{split}
					&\begin{pmatrix}
						\varepsilon E_{3_{{\cdot,j,k}}}^{[1]}\\
						\mu H_{2_{{\cdot,j,k}}}^{[1]}
					\end{pmatrix}-\begin{pmatrix}
						\varepsilon E_{3_{{\cdot,j,k}}}^{n}\\
						\mu H_{2_{{\cdot,j,k}}}^{n}
					\end{pmatrix}=\frac{\tau}{h_x}\mathcal{A}^{-1}\mathcal{B}\begin{pmatrix}
						E_{3_{{\cdot,j,k}}}^{[1]+n}\\
						H_{2_{{\cdot,j,k}}}^{[1]+n}
					\end{pmatrix},\\
					&\begin{pmatrix}
						\varepsilon E_{1_{{i,\cdot,k}}}^{[1]}\\
						\mu H_{3_{{i,\cdot,k}}}^{[1]}
					\end{pmatrix}-\begin{pmatrix}
						\varepsilon E_{1_{{i,\cdot,k}}}^{n}\\
						\mu H_{3_{{i,\cdot,k}}}^{n}
					\end{pmatrix}
					=\frac{\tau}{h_y}\mathcal{A}^{-1}\mathcal{B}\begin{pmatrix}
						E_{1_{{i,\cdot,k}}}^{[1]+n}\\
						H_{3_{{i,\cdot,k}}}^{[1]+n}
					\end{pmatrix},\\
					&\begin{pmatrix}
						\varepsilon E_{2_{{i,j,\cdot}}}^{[1]}\\
						\mu H_{1_{{i,j,\cdot}}}^{[1]}
					\end{pmatrix}-\begin{pmatrix}
						\varepsilon E_{2_{{i,j,\cdot}}}^{n}\\
						\mu H_{1_{{i,j,\cdot}}}^{n}
					\end{pmatrix}=\frac{\tau}{h_z}\mathcal{A}^{-1}\mathcal{B}\begin{pmatrix}
						E_{2_{{i,j,\cdot}}}^{[1]+n}\\
						H_{1_{{i,j,\cdot}}}^{[1]+n}
					\end{pmatrix}.
				\end{split}
			\end{equation}
			
			{\bf Stage 2.}  Compute intermediate variables ${\bf E}^{[2]},{\bf H}^{[2]}$ from ${\bf E}^{[1]}$ and ${\bf H}^{[1]}$:
			\begin{equation}\label{101}
				\begin{split}
					&\begin{pmatrix}
						\varepsilon E_{2_{{\cdot,j,k}}}^{[2]}\\
						\mu H_{3_{{\cdot,j,k}}}^{[2]}
					\end{pmatrix}-\begin{pmatrix}
						\varepsilon E_{2_{{\cdot,j,k}}}^{[1]}\\
						\mu H_{3_{{\cdot,j,k}}}^{[1]}
					\end{pmatrix}=-\frac{\tau}{h_x}\mathcal{A}^{-1}\mathcal{B}\begin{pmatrix}
						E_{2_{{\cdot,j,k}}}^{[2]+[1]}\\
						H_{3_{{\cdot,j,k}}}^{[2]+[1]}
					\end{pmatrix},\\
					&\begin{pmatrix}
						\varepsilon E_{3_{{i,\cdot,k}}}^{[2]}\\
						\mu H_{1_{{i,\cdot,k}}}^{[2]}
					\end{pmatrix}-\begin{pmatrix}
						\varepsilon E_{3_{{i,\cdot,k}}}^{[1]}\\
						\mu H_{1_{{i,\cdot,k}}}^{[1]}
					\end{pmatrix}=-\frac{\tau}{h_y}\mathcal{A}^{-1}\mathcal{B}\begin{pmatrix}
						E_{3_{{i,\cdot,k}}}^{[2]+[1]}\\
						H_{1_{{i,\cdot,k}}}^{[2]+[1]}
					\end{pmatrix},\\
					&\begin{pmatrix}
						\varepsilon E_{1_{{i,j,\cdot}}}^{[2]}\\
						\mu H_{2_{{i,j,\cdot}}}^{[2]}
					\end{pmatrix}-\begin{pmatrix}
						\varepsilon E_{1_{{i,j,\cdot}}}^{[1]}\\
						\mu H_{2_{{i,j,\cdot}}}^{[1]}
					\end{pmatrix}=-\frac{\tau}{h_z}\mathcal{A}^{-1}\mathcal{B}\begin{pmatrix}
						E_{1_{{i,j,\cdot}}}^{[2]+[1]}\\
						H_{2_{{i,j,\cdot}}}^{[2]+[1]}
					\end{pmatrix}.
				\end{split}
			\end{equation}
			
			{\bf Stage 3.}  Compute ${\bf E}^{n+1},{\bf H}^{n+1}$ from intermediate variables ${\bf E}^{[2]}$ and ${\bf H}^{[2]}$:
			\begin{equation}\label{22}
				\begin{pmatrix}
					{\bf E}_{{i,j,k}}^{n+1}\\
					{\bf H}_{{i,j,k}}^{n+1}
				\end{pmatrix}=S^n((\Delta W)_{i,j,k}^n)\begin{pmatrix}
					{\bf E}_{{i,j,k}}^{[2]}\\
					{\bf H}_{{i,j,k}}^{[2]}
				\end{pmatrix}. 
			\end{equation}
			
From \eqref{100} and \eqref{101}, we observe that \textbf{Splitting Method I} proves straightforward to implement. At each time step, it requires solving only one-dimensional linear algebraic systems while completely eliminating the need for storing intermediate variables, thereby minimizing memory requirements.

Through theoretical analysis and numerical experiments in subsequent sections, we demonstrate that \textbf{Splitting Method I} preserves the discrete energy conservation law. Applying the second splitting strategy from Section \ref{sect_3} yields another effective splitting approach.
			
			{\bf (2) Splitting method II}
			
			{\bf Stage 1.}  Compute intermediate variables ${\bf E}^{[1]},{\bf H}^{[1]}$ from ${\bf E}^{n}$ and ${\bf H}^{n}$:
			\begin{equation}\label{102}
				\begin{split}
					&E_{1_{{\cdot,j,k}}}^{[1]}=E_{1_{{\cdot,j,k}}}^{n},\quad H_{1_{{\cdot,j,k}}}^{[1]}=H_{1_{{\cdot,j,k}}}^{n}\\
					&\begin{pmatrix}
						\varepsilon E_{2_{{\cdot,j,k}}}^{[1]}\\
						\mu H_{3_{{\cdot,j,k}}}^{[1]}
					\end{pmatrix}-\begin{pmatrix}
						\varepsilon E_{2_{{\cdot,j,k}}}^{n}\\
						\mu H_{3_{{\cdot,j,k}}}^{n}
					\end{pmatrix}
					=-\frac{\tau}{h_x}\mathcal{A}^{-1}\mathcal{B}\begin{pmatrix}
						E_{2_{{\cdot,j,k}}}^{[1]+n}\\
						H_{3_{{\cdot,j,k}}}^{[1]+n}
					\end{pmatrix},\\
					&\begin{pmatrix}
						\varepsilon E_{3_{{\cdot,j,k}}}^{[1]}\\
						\mu H_{2_{{\cdot,j,k}}}^{[1]}
					\end{pmatrix}-\begin{pmatrix}
						\varepsilon E_{3_{{\cdot,j,k}}}^{n}\\
						\mu H_{2_{{\cdot,j,k}}}^{n}
					\end{pmatrix}
					=\frac{\tau}{h_x}\mathcal{A}^{-1}\mathcal{B}\begin{pmatrix}
						E_{3_{{\cdot,j,k}}}^{[1]+n}\\
						H_{2_{{\cdot,j,k}}}^{[1]+n}
					\end{pmatrix}.
				\end{split}
			\end{equation}
			
			{\bf Stage 2.}  Compute intermediate variables ${\bf E}^{[2]},{\bf H}^{[2]}$ from ${\bf E}^{[1]}$ and ${\bf H}^{[1]}$:
			\begin{equation}\label{103}
				\begin{split}
					&E_{2_{{i,\cdot,k}}}^{[2]}=E_{2_{{i,\cdot,k}}}^{[1]},\quad H_{2_{{i,\cdot,k}}}^{[2]}=H_{2_{{i,\cdot,k}}}^{[1]},\\
					&\begin{pmatrix}
						\varepsilon E_{3_{{i,\cdot,k}}}^{[2]}\\
						\mu H_{1_{{i,\cdot,k}}}^{[2]}
					\end{pmatrix}-\begin{pmatrix}
						\varepsilon E_{3_{{i,\cdot,k}}}^{[1]}\\
						\mu H_{1_{{i,\cdot,k}}}^{[1]}
					\end{pmatrix}
					=-\frac{\tau}{h_y}\mathcal{A}^{-1}\mathcal{B}\begin{pmatrix}
						E_{3_{{i,\cdot,k}}}^{[2]+[1]}\\
						H_{1_{{i,\cdot,k}}}^{[2]+[1]}
					\end{pmatrix},\\
					&\begin{pmatrix}
						\varepsilon E_{1_{{i,\cdot,k}}}^{[2]}\\
						\mu H_{3_{{i,\cdot,k}}}^{[2]}
					\end{pmatrix}-\begin{pmatrix}
						\varepsilon _{1_{{i,\cdot,k}}}^{[1]}\\
						\mu H_{3_{{i,\cdot,k}}}^{[1]}
					\end{pmatrix}
					=\frac{\tau}{h_y}\mathcal{A}^{-1}\mathcal{B}\begin{pmatrix}
						E_{1_{{i,\cdot,k}}}^{[2]+[1]}\\
						H_{3_{{i,\cdot,k}}}^{[2]+[1]}
					\end{pmatrix},
				\end{split}
			\end{equation}
			
			{\bf Stage 3.}  Compute intermediate variables ${\bf E}^{[3]},{\bf H}^{[3]}$ from ${\bf E}^{[2]}$ and ${\bf H}^{[2]}$:
			\begin{equation}\label{104}
				\begin{split}
					&E_{3_{{i,j,\cdot}}}^{[3]}=E_{3_{{i,j,\cdot}}}^{[2]},\quad H_{3_{{i,j,\cdot}}}^{[3]}=H_{3_{{i,j,\cdot}}}^{[2]}\\
					&\begin{pmatrix}
						\varepsilon E_{1_{{i,j,\cdot}}}^{[3]}\\
						\mu H_{2_{{i,j,\cdot}}}^{[3]}
					\end{pmatrix}-\begin{pmatrix}
						\varepsilon E_{1_{{i,j,\cdot}}}^{[2]}\\
						\mu H_{2_{{i,j,\cdot}}}^{[2]}
					\end{pmatrix}
					=-\frac{\tau}{h_z}\mathcal{A}^{-1}\mathcal{B}\begin{pmatrix}
						E_{1_{{i,j,\cdot}}}^{[3]+[2]}\\
						H_{2_{{i,j,\cdot}}}^{[3]+[2]}
					\end{pmatrix},\\
					&\begin{pmatrix}
						\varepsilon E_{2_{{i,j,\cdot}}}^{[3]}\\
						\mu H_{1_{{i,j,\cdot}}}^{[3]}
					\end{pmatrix}-\begin{pmatrix}
						\varepsilon E_{2_{{i,j,\cdot}}}^{[2]}\\
						\mu H_{1_{{i,j,\cdot}}}^{[2]}
					\end{pmatrix}
					=\frac{\tau}{h_z}\mathcal{A}^{-1}\mathcal{B}\begin{pmatrix}
						E_{2_{{i,j,\cdot}}}^{[3]+[2]}\\
						H_{1_{{i,j,\cdot}}}^{[3]+[2]}
					\end{pmatrix}.
				\end{split}
			\end{equation}
			
			{\bf Stage 4.}  Compute ${\bf E}^{n+1},{\bf H}^{n+1}$ from intermediate variables ${\bf E}^{[3]}$ and ${\bf H}^{[3]}$:
			\begin{equation}
				\begin{pmatrix}
					{\bf E}_{{i,j,k}}^{n+1}\\
					{\bf H}_{{i,j,k}}^{n+1}
				\end{pmatrix}=S^n((\Delta W)_{i,j,k}^n)\begin{pmatrix}
					{\bf E}_{{i,j,k}}^{[3]}\\
					{\bf H}_{{i,j,k}}^{[3]}
				\end{pmatrix}.
			\end{equation}
			
			For this four-stages method, we will prove in next subsection that it conserves energy.
			
			\subsection{Conservation properties}\label{Conservation properties}
			In this part, we derive the discrete energy conservation property for the operator splitting methods which proposed in the previous subsection.
			\begin{theorem}\label{TT1}
				For any $n\geq0$, {\bf Splitting method I} possesses the discrete energy conservation law
				\begin{align}
					\varepsilon\|{\bf E}^{n+1}\|^2+\mu\|{\bf H}^{n+1}\|^2=\varepsilon\|{\bf E}^{n}\|^2+\mu\|{\bf H}^{n}\|^2,\quad{\mathbb P}\text{-}a.s.,
				\end{align}
				here and in what follows, $\|{\bf f}\|^2=\|f_1\|_{l^2}^2+\|f_2\|_{l^2}^2+\|f_3\|_{l^2}^2$ for any ${\bf f}=(f_1,f_2,f_3)$.
			\end{theorem}
			\begin{proof}
				We multiply both sides of equations \eqref{100} with $h_yh_zP_{\cdot,j,k}^{[1]+n}$, $h_xh_zQ_{i,\cdot,k}^{[1]+n}$ and $h_xh_yR_{i,j,\cdot}^{[1]+n}$, here $P_{\cdot,j,k}^{[1]+n}=\left(E_{3_{\cdot,j,k}}^{[1]+n},H_{2_{\cdot,j,k}}^{[1]+n}\right)^T$, $Q_{i,\cdot,k}^{[1]+n}=\left(E_{1_{i,\cdot,k}}^{[1]+n},H_{3_{i,\cdot,k}}^{[1]+n}\right)^T$ and $R_{i,j,\cdot}^{[1]+n}=\left(E_{2_{i,j,\cdot}}^{[1]+n},H_{1_{i,j,\cdot}}^{[1]+n}\right)^T$ respectively. we obtain
				\begin{equation*}
					\begin{split}
						&h_xh_yh_z\left\langle\begin{pmatrix}
							\varepsilon E_{3_{{\cdot,j,k}}}^{[1]}-\varepsilon E_{3_{{\cdot,j,k}}}^{n}\\
							\mu H_{2_{{\cdot,j,k}}}^{[1]}-\mu H_{2_{{\cdot,j,k}}}^{n}
						\end{pmatrix},P_{\cdot,j,k}^{[1]+n}\right\rangle=\tau{h_y}{h_z}\left\langle\mathcal{A}^{-1}\mathcal{B}P_{\cdot,j,k}^{[1]+n},P_{\cdot,j,k}^{[1]+n}\right\rangle,\\
						&h_xh_yh_z\left\langle\begin{pmatrix}
							\varepsilon E_{1_{i,\cdot,k}}^{[1]}-\varepsilon E_{1_{i,\cdot,k}}^{[n]}\\
							\mu H_{3_{i,\cdot,k}}^{[1]}-\mu H_{3_{i,\cdot,k}}^{n}
						\end{pmatrix},Q_{i,\cdot,k}^{[1]+n}\right\rangle=\tau{h_x}{h_z}\left\langle\mathcal{A}^{-1}\mathcal{B}Q_{i,\cdot,k}^{[1]+n},Q_{i,\cdot,k}^{[1]+n}\right\rangle,\\
						&h_xh_yh_z\left\langle\begin{pmatrix}
							\varepsilon E_{2_{i,j,\cdot}}^{[1]}-\varepsilon E_{2_{i,j,\cdot}}^{n}\\
							\mu H_{1_{i,j,\cdot}}^{[1]}-\mu H_{1_{i,j,\cdot}}^{n}
						\end{pmatrix},R_{i,j,\cdot}^{[1]+n}\right\rangle=\tau{h_x}{h_y}\left\langle\mathcal{A}^{-1}\mathcal{B}R_{i,j,\cdot}^{[1]+n},R_{i,j,\cdot}^{[1]+n}\right\rangle.
					\end{split}
				\end{equation*}
				Summing over all terms in the above equations, and adding them together, From Lemma \ref{lemma_1}, it follows that:
				\begin{align}
					\varepsilon\|{\bf E}^{[1]}\|^2+\mu\|{\bf H}^{[1]}\|^2=\varepsilon\|{\bf E}^{n}\|^2+\mu\|{\bf H}^{n}\|^2,\quad{\mathbb P}\text{-}a.s.
				\end{align}
				Similarly, from \eqref{101}, we have
				\begin{align}
					\varepsilon\|{\bf E}^{[2]}\|^2+\mu\|{\bf H}^{[2]}\|^2=\varepsilon\|{\bf E}^{[1]}\|^2+\mu\|{\bf H}^{[1]}\|^2,\quad{\mathbb P}\text{-}a.s.
				\end{align}
				For \eqref{22}, note that even in the stochastic case $S^n$ is a unitary $C_0$-group on $\mathbb{H}$, thus we have
				\begin{equation}
					\left\|\begin{pmatrix}
						{\bf E}_{{i,j,k}}^{n+1}\\
						{\bf H}_{{i,j,k}}^{n+1}
					\end{pmatrix}\right\|_{\mathbb{H}}=\left\|\begin{pmatrix}
						{\bf E}_{{i,j,k}}^{[2]}\\
						{\bf H}_{{i,j,k}}^{[2]}
					\end{pmatrix}\right\|_{\mathbb{H}},\quad{\mathbb P}\text{-}a.s.
				\end{equation}
				Finally, combining the above equations gives the desired result. The proof is complete.
			\end{proof}
			
			Similarly, the discrete energy conservation property holds for the {\bf Splitting method II}, which is stated in the following theorem.
			\begin{theorem}\label{TT2}
				For any $n\geq0$, {\bf Splitting method II} possesses the discrete energy conservation law
				\begin{align}
					\varepsilon\|{\bf E}^{n+1}\|^2+\mu\|{\bf H}^{n+1}\|^2=\varepsilon\|{\bf E}^{n}\|^2+\mu\|{\bf H}^{n}\|^2,\quad{\mathbb P}\text{-}a.s.
				\end{align}
			\end{theorem}
			
\section{Numerical experiments}\label{Numerical experiments2}
This section verifies the energy conservation of Splitting Method I and Splitting Method II through numerical examples. In the experiment, we take uniform spatial stepsizes $h_x = h_y = h_z = h$. The $Q$-Wiener process $\{W(t), t \geq 0\}$ has the following Karhunen-Loève expansion, with $\{\eta_{m,\ell,q}\}_{m,\ell,q\in\mathbb{N}}$ and $\{e_{m,\ell,q}(x,y,z)\}_{m,\ell,q\in\mathbb{N}}$ chosen as:
$$\eta_{m,\ell,q} = \frac{1}{\sqrt{m^{3} + \ell^{3} + q^{3}}},$$
and
$$e_{m,\ell,q}(x,y,z) = 2\sqrt{2} \sin \left(m \pi x\right) \sin \left(\ell \pi y\right) \sin \left(q \pi z\right).$$
The noise increment is defined as:
$$(\Delta W)_{i,j,k}^{n} := W_{i,j,k}^{n+1} - W_{i,j,k}^{n} = 2 \sqrt{2\tau} \sum_{m,\ell,q=1}^{M} \frac{1}{\sqrt{m^{3} + \ell^{3} + q^{3}}} \sin \left(m \pi x_{i}\right) \sin \left(\ell \pi y_{j}\right) \sin \left(q \pi z_{k}\right) \xi_{m,\ell,q}^{n},$$
where $\left\{\xi_{m,\ell,q}^{n}\right\}$ are independent $\mathcal{N}(0,1)$ random variables. Since $(\Delta W)_{i,j,k}^{n}$ exhibits negligible variation for $M > 10$, we truncate at $M = 10$ to enhance computational efficiency.
\subsection{3D Maxwell Numerical Example}
The numerical experiment considers the 3D stochastic Maxwell equations \eqref{sto_maxwell_equations} on the rectangular domain $D = \left[0, \frac{1}{2}\right]^3$ with the following initial conditions:
$$\left\{ 
\begin{array}{l} 
E_1(x,y,z,0) = \cos(4\pi(x + y + z)),  \\
E_2(x,y,z,0) = -2E_1(x,y,z,0), \\
E_3(x,y,z,0) = E_1(x,y,z,0), 
\end{array}
\right.$$
and
$$\left\{ 
\begin{array}{l} 
H_1(x,y,z,0) = \sqrt{3} E_1(x,y,z,0), \\
H_2(x,y,z,0) = 0, \\
H_3(x,y,z,0) = -\sqrt{3} E_3(x,y,z,0).
\end{array}
\right.$$

According to Theorems \ref{TT1} and \ref{TT2}, Splitting Method I and Splitting Method II preserve the discrete energy conservation law almost surely. To verify this property, we fix $\epsilon = 1$, $\mu = 1$, $\lambda = 0, 0.1, 1, 10$, time step $\tau = \frac{1}{2^5}$, spatial step size $h = \frac{1}{50}$, and truncation parameter $M = 10$. We then plot the evolution of discrete energy $\mathcal{H}(Z(t))$ using numerical solutions from both splitting methods at time nodes $t_i = i\tau$ over $T = 10$.
\begin{figure}
	\begin{center}	
         \subfigure[]{
			\begin{minipage}[t]{0.45\linewidth}
				\includegraphics[width=1\textwidth]{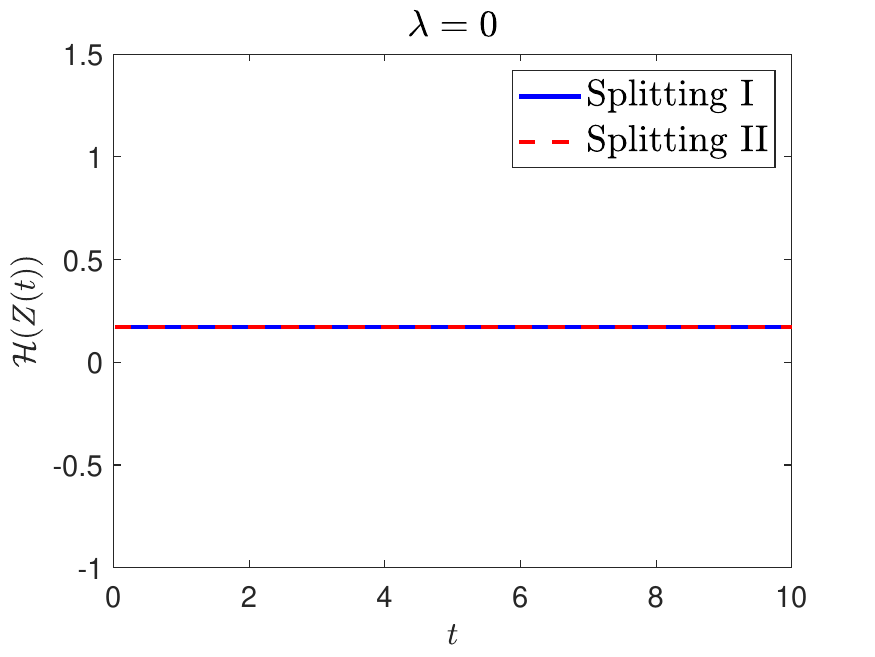}
			\end{minipage}
        
		}
		\subfigure[]{
			\begin{minipage}[t]{0.45\linewidth}
				\includegraphics[width=1\textwidth]{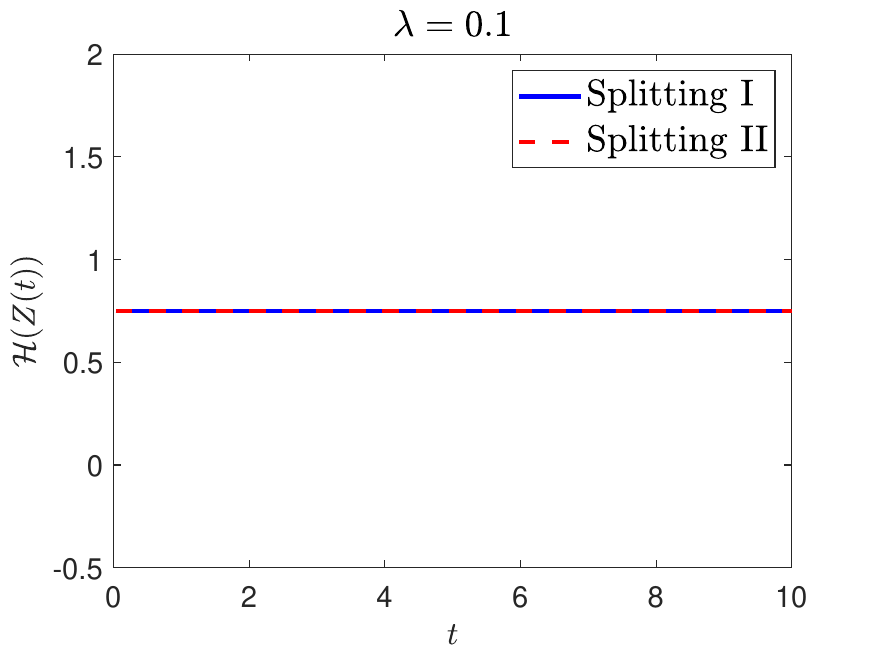}
			\end{minipage}
		}
        \subfigure[]{
			\begin{minipage}[t]{0.45\linewidth}
				\includegraphics[width=1\textwidth]{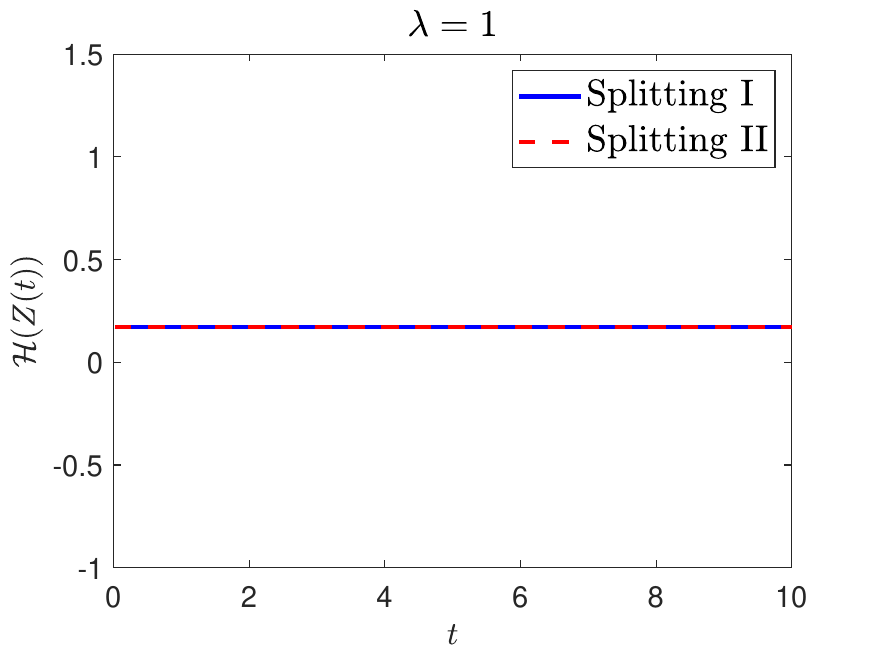}
			\end{minipage}
		}
        \subfigure[]{
			\begin{minipage}[t]{0.45\linewidth}
				\includegraphics[width=1\textwidth]{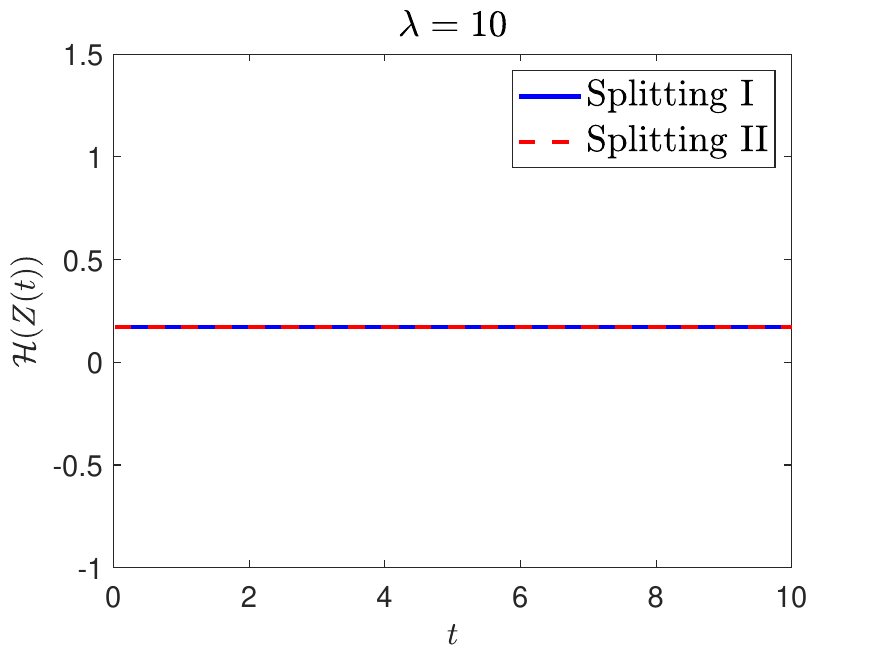}
			\end{minipage}
		}
	\end{center}
	\caption{
		  The discrete energy until $T = 10$ with $\tau = \frac{1}{2^5}$, $h_x=h_y=h_z = \frac{1}{50}$.
		\label{Energy_lambda}}
\end{figure}	
 In Figure \ref{Energy_lambda}, the solid line represents the discrete energy evolution of Stochastic Splitting Method I, and the dashed line corresponds to Stochastic Splitting Method II. The figure clearly shows energy evolution approximating horizontal lines over time, indicating that discrete energies for both splitting methods remain nearly constant and match the initial energy. This observation aligns with the theoretical results established in previous theorems. Furthermore, despite testing different noise intensity parameters $\lambda = 0, 0.1, 1, 10$, all discrete energy curves maintain their approximately horizontal profiles. These results demonstrate that both numerical methods preserve the energy conservation property unconditionally, consistent with theoretical predictions.
\begin{figure}
	\begin{center}	
         \subfigure[]{
			\begin{minipage}[t]{0.4\linewidth}
				\includegraphics[width=1\textwidth]{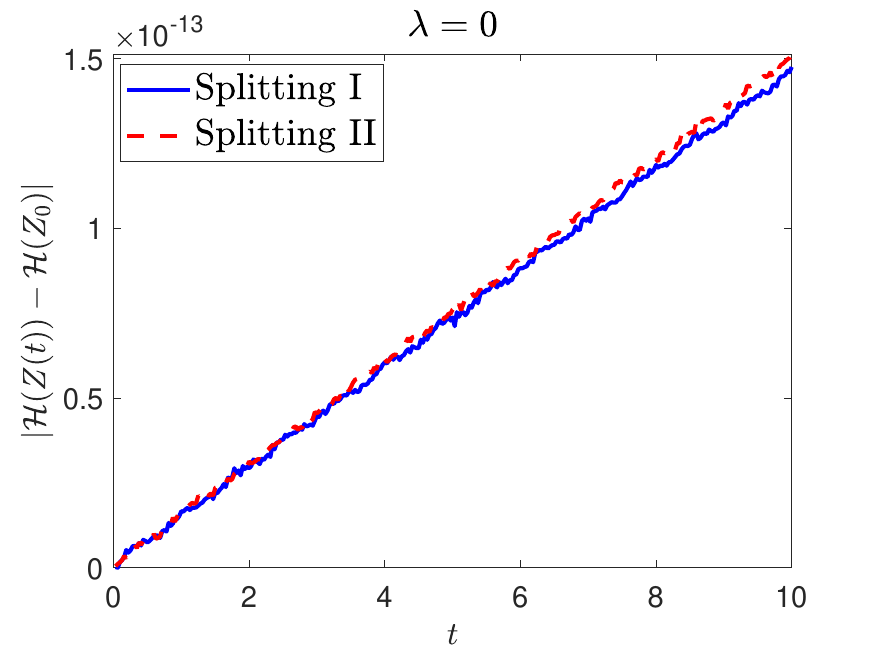}
			\end{minipage}
		}
		\subfigure[]{
			\begin{minipage}[t]{0.4\linewidth}
				\includegraphics[width=1\textwidth]{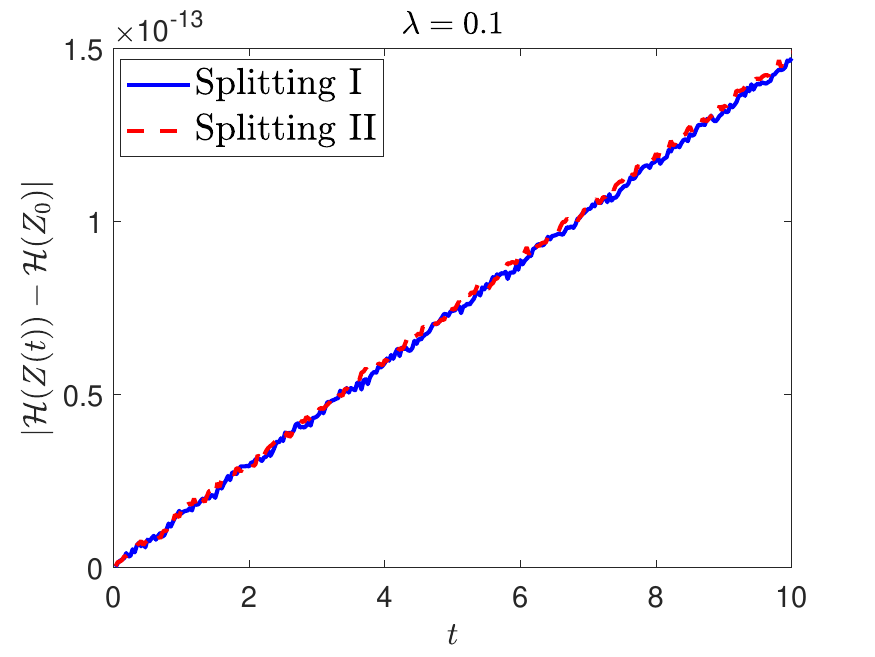}
			\end{minipage}
		}
        \subfigure[]{
			\begin{minipage}[t]{0.4\linewidth}
				\includegraphics[width=1\textwidth]{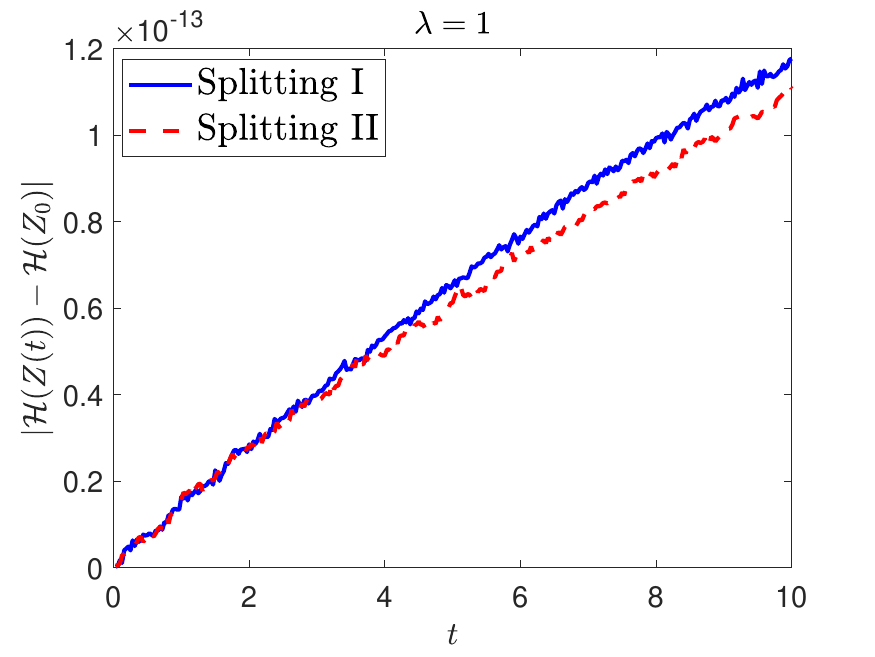}
			\end{minipage}
		}
        \subfigure[]{
			\begin{minipage}[t]{0.4\linewidth}
				\includegraphics[width=1\textwidth]{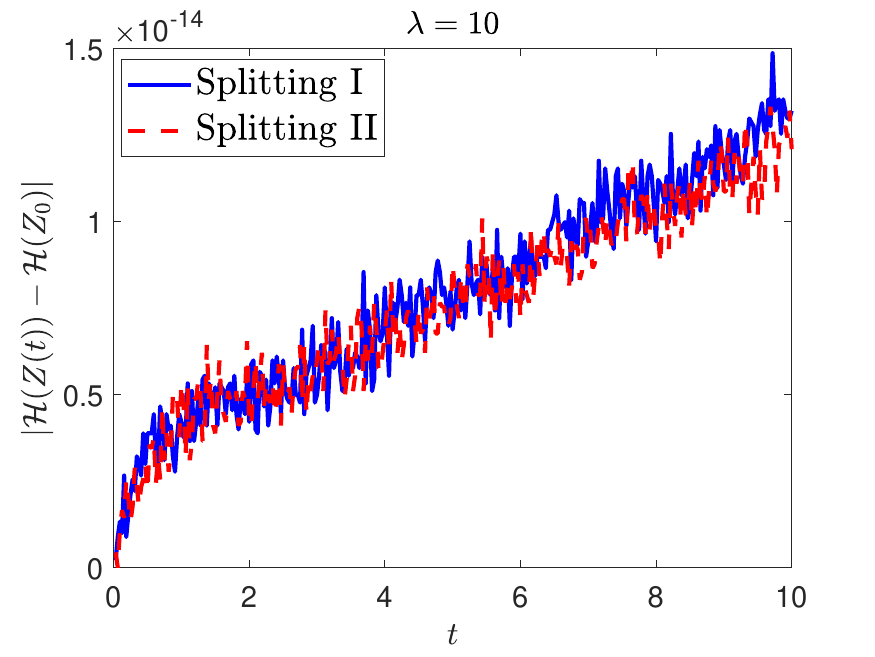}
			\end{minipage}
		}
	\end{center}
	\caption{
		The errors of discrete energy until $T = 10$ with $\tau = \frac{1}{2^5}$, $h_x=h_y=h_z = \frac{1}{50}$.
		\label{Energy_error_lambda}}
\end{figure}		

To further investigate the magnitude of error between discrete energy and initial energy, Figure \ref{Energy_error_lambda} plots the global energy error at $T=10$. Using fixed parameters $\lambda = 0, 0.1, 1, 10$, $\tau = \frac{1}{2^5}$, $h = \frac{1}{50}$, and $M=10$, we observe that the global errors in the discrete energy conservation law consistently reach $10^{-13}$ magnitude. Due to the accumulated rounding errors of the computer and the approximations of the iterative solver, the energy error exhibits a linear growth. By excluding these computational errors, it can be considered that the splitting methods still preserve the discrete energy conservation law.

To further verify that the energy conservation properties of both splitting methods remain unaffected by varying noise trajectories, we perform simulations over $T=10$ with three distinct random paths. By comparing the energy evolution of numerical solutions against initial energy, we test method conservation properties. 
\begin{figure}
	\begin{center}	
         \subfigure[]{
			\begin{minipage}[t]{0.4\linewidth}
				\includegraphics[width=1\textwidth]{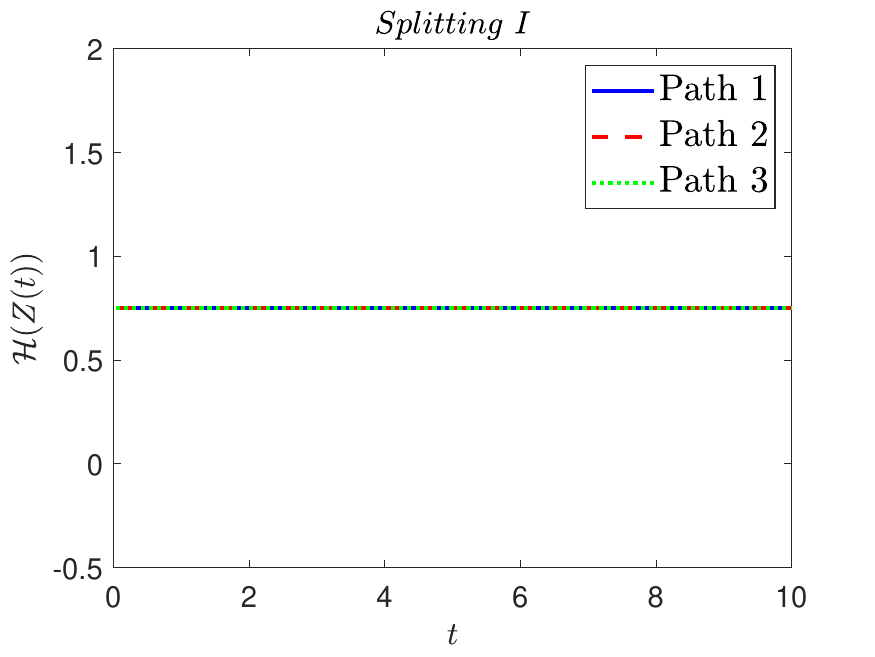}
			\end{minipage}
		}
		\subfigure[]{
			\begin{minipage}[t]{0.4\linewidth}
				\includegraphics[width=1\textwidth]{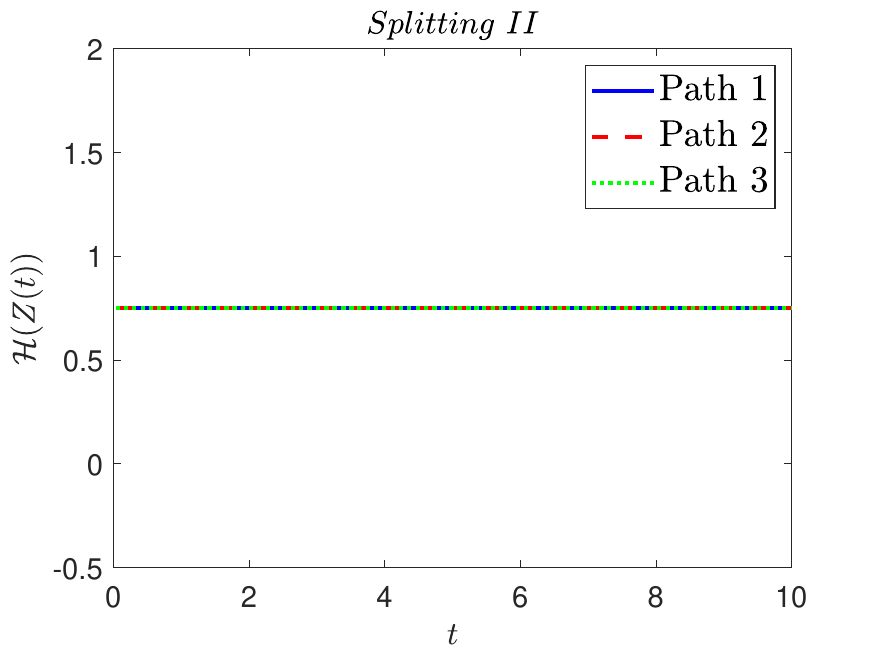}
			\end{minipage}
		}
	\end{center}
	\caption{
		The discrete energy until $T = 10$ with $\tau = \frac{1}{2^5}$, $h = \frac{1}{50}$.
		\label{Energy_path}}
\end{figure}		
\begin{figure}
	\begin{center}	
         \subfigure[]{
			\begin{minipage}[t]{0.4\linewidth}
				\includegraphics[width=1\textwidth]{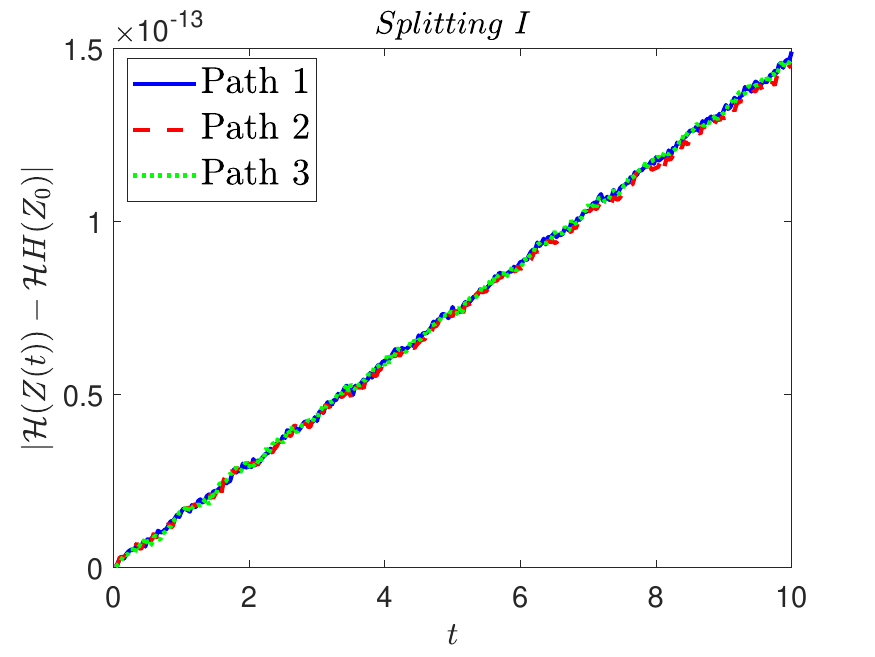}
			\end{minipage}
		}
		\subfigure[]{
			\begin{minipage}[t]{0.4\linewidth}
				\includegraphics[width=1\textwidth]{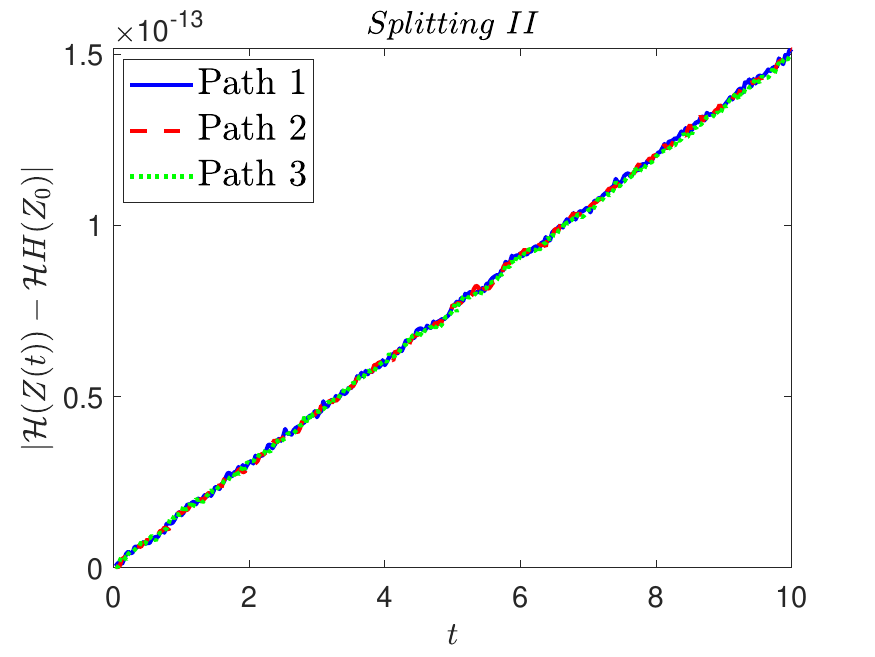}
			\end{minipage}
		}
	\end{center}
	\caption{
		The discrete energy error until $T = 10$ with $\tau = \frac{1}{2^5}$, $h = \frac{1}{50}$.
		\label{Energy_errror_path}}
\end{figure}	
Figure \ref{Energy_path} demonstrates that both Splitting Methods I and II maintain strict energy conservation across all tested paths, with energy curves parallel to the $t$-axis throughout computations. Figure \ref{Energy_errror_path} also shows that the energy errors are basically the same under the three different noise paths. These results confirm experimental independence from noise path selection, highlighting method stability and reliability.

Finally, we numerically investigate the temporal mean-square convergence orders of Splitting Methods I and II. We fix parameters $\epsilon = 1$, $\mu = 1$, $\lambda = 0.1$, $h = \frac{1}{50}$, and $T = \frac{1}{4}$, using time steps $\tau = 2^{-4}, 2^{-5}, 2^{-6}, 2^{-7}, 2^{-8}$. As the exact solution for this initial value is unknown, we employ the numerical solution computed by Splitting Method I with $\tau = 2^{-9}$ as the reference solution. The $L^2$-error is defined as $\|Error\|_{L^2} = \left(\mathbb{E}\|e\|_{L^2}^2\right)^{\frac{1}{2}}$, where $\|e\|_{L^2}^2$ is given by:
\begin{align*}
\|e\|_{L^2}^2=\Delta x\Delta y\Delta z\sum_{i,j,k}\sum_{m=1}^{3}\left[\left((E_m)_{i,j,k}^N-(E_m^{ref})_{i,j,k}^N\right)^2+\left((H_m)_{i,j,k}^N-(H_m^{ref})_{i,j,k}^N\right)^2\right].
\end{align*}
\begin{table}[!htb]
	\begin{center}
		\caption{The order of error until $T = \frac{1}{4}$ with $h= \frac{1}{50}$.}\label{Order}
		\label{tab:table1}
		\begin{tabular}{l|ll|ll}\toprule
			&\multicolumn{2}{c|}{Splitting method I} & \multicolumn{2}{c}{Splitting method II}\\ 
			\cline{2-5} 
			&$\|Error\|_{L^2}$ & Order &$\|Error\|_{L^2}$ & Order \\\hline
                $\tau=1/2^4$& $4.72E-1$&$-$&$6.64E-1$&$-$\\
			$\tau=1/2^5$ & $1.72E-1$ & $1.46$ & $3.34E-1$&$0.97$\\
			$\tau=1/2^6$ & $7.94E-2$ & $1.11$ & $1.71E-1$&$0.98$\\
			$\tau=1/2^7$& $4.68E-2$ & $0.76$ & $7.91E-2$&$0.92$\\
			$\tau=1/2^8$& $2.40E-2$&$0.96$&$4.67E-2$&$0.96$\\
			\bottomrule
		\end{tabular}
	\end{center}
\end{table}
As shown in Table \ref{Order}, both splitting methods strong order of converge is approximately 1 in the time direction.
		
\section{Conclusions}
We successfully develop and analyze two novel operator splitting methods, Splitting I and Splitting II, for structure-preserving numerical solutions of three-dimensional stochastic Maxwell equations driven by multiplicative noise. It is rigorously proven that both splitting methods preserve the discrete electromagnetic energy conservation law. Specifically, it is confirmed that the midpoint method for the deterministic part maintains the quadratic invariant associated with the anti-symmetric operator, while the exact integration of the stochastic part exhibits unitarity. The combination of these two ensures the strict energy conservation of the entire algorithm. Numerical experiments fully validate the theoretical results: under different noise intensities and random paths, the discrete energy remains constant over time, and the global energy error reaches the magnitude of machine precision, verifying the excellent conservation performance of the methods in practice. Convergence studies further show that both splitting methods achieve first-order convergence in the mean square sense in the time direction. However, both the theoretical research and numerical experiments in this paper assume that the noise is a Q-Wiener process. For other types of noise, it is necessary to further verify whether the existing methods still possess discrete energy conservation. Additionally, for other boundary conditions, such as absorbing boundaries or impedance boundaries, further research is required to determine whether the methods satisfy the discrete conservation laws.
\section*{Acknowledgments}
The authors would like to express their appreciation to the referees for their useful comments and the editors. Liying Zhang is supported by the National Natural Science Foundation of China (No. 11601514 and No. 11971458), the Fundamental Research Funds for the Central Universities (No. 2023ZKPYL02 and No. 2023JCCXLX01) and the Yueqi Youth Scholar Research Funds for the China University of Mining and Technology-Beijing (No. 2020YQLX03), 2025 Basic Sciences Initiative in Mathematics and Physics. Lihai Ji is supported by the National Natural Science Foundation of China (No. 12171047).



\begin{thebibliography}{99}

\bibitem{CC1} C. C. Chen, J. L. Hong and L. Y. Zhang, Preservation of physical properties of stochastic Maxwell equations with additive noise via stochastic multi-symplectic methods, J. Comput. Phys., 306(2016), 500--519. https://doi.org/10.1016/j.jcp.2015.11.052.

\bibitem{CC2} C. C. Chen, J. L. Hong and L. H. Ji, Mean-square convergence of a semiDiscrete scheme for stochastic {M}axwell equations, SIAM J. Numer. Anal., 57 (2019), 728--750. https://doi.org/10.1137/18M1170431

\bibitem{CC3}C. C. Chen and J. L. Hong and L. H. Ji, Runge-{K}utta semidiscretizations for stochastic {M}axwell equations with additive noise, SIAM J. Numer. Anal., 57 (2019), 702--727. https://doi.org/10.1137/18M1193372

\bibitem{DD1} D. Cohen, J. B. Cui, J. L. Hong and L. Y. Sun, Exponential integrators for stochastic Maxwell's equations driven by It\^o noise, J. Comput. Phys., 410 (2020), 109382, https://doi.org/10.1016/j.jcp.2020.109382

\bibitem{JJ1} J. L. Hong, L. H. Ji, L. Y. Zhang and J. X. Cai, An energy-conserving method for stochastic Maxwell equations with multiplicative noise, J. Comput. Phys., 350 (2017), 216--229. https://doi.org/10.1016/j.jcp.2017.09.030

\bibitem{JJ2} J. X. Cai, J. L. Hong, Y. S. Wang and Y. Z. Gong, Two energy-conserved splitting methods for three-dimensionatime-domain Maxwell's equations and the convergence analysis, SIAM J. NUMER. ANAL., 53 (2015), 1918--1940. https://doi.org/10.1137/140971609

\bibitem{EE1} K. J. Engel and N. Rainer, One-parameter semigroups for linear evolution equations, Springer-Verlag, New York, 2000. https://link.springer.com/book/10.1007/b97696

\bibitem{LL1}L. Kurt and T. Schäfer, Propagation of ultra-short solitons in stochastic Maxwell's equations, J. Math. Phys., 55 (2014), 011503. https://doi.org/10.1063/1.4859815

\bibitem{ZZ1} L. Y. Zhang, C. C. Chen, J. L. Hong and L. H. Ji, A review on stochastic multi-symplectic methods for stochastic {M}axwell equations, Commun. Appl. Math. Comput., 1 (2019), 467--501. https://doi.org/10.1007/s42967-019-00017-w

\bibitem{MM1}M. Francoeur and M. P. Mengüç, Role of fluctuational electrodynamics in near-field radiative heat transfer, J. Quant. Spectrosc. Radiat. Transf., 109 (2008), 280--293. https://doi.org/10.1016/j.jqsrt.2007.08.017

\bibitem{SS1}S. M. Rytov, Y. A. Kravtsov and V. I. Tatarskii, Principles of Statistical Radiophysics 3: Elements of Random Fields, Springer-Verlag, Berlin, 1989. https://link.springer.com/book/9783642726873

\bibitem{WW1} W. B. Chen, X. J. Li and D. Liang, Energy-conserved splitting FDTD methods for Maxwell equations, Numer. Math., 108 (2008), 445--485. https://doi.org/10.1007/s00211-007-0123-9

\bibitem{WW2}W. B. Chen, X. J. Li and D. Liang, Energy-conserved splitting FDTD methods for Maxwell equations in three dimensions. SIAM J. Numer. Anal., 48 (2010), 1530–1554. https://doi.org/10.1137/090765857




\end{thebibliography}
\end{document}